\documentclass{amsart}
\usepackage{amssymb,amsmath,amsbsy}
\usepackage{url}
\usepackage{booktabs}
\usepackage{tikz}
\usepackage{nicefrac}
\usetikzlibrary{shapes.misc}

\newcommand{\ZZ}{\mathbb{Z}}
\newcommand{\FF}{\mathbb{F}}
\newcommand{\Fp}{\FF_p}
\newcommand{\Fptwo}{\FF_{p^2}}

\newcommand{\Fq}{\FF_q}

\newcommand{\inkron}[2]{\genfrac {(}{)}{0.9pt}{}{#1}{#2}}
\newcommand{\M}{\textsf{M}}
\newcommand{\OO}{\mathcal{O}}
\def\llog{\operatorname{llog}}
\newcommand{\logpow}[2]{{\hspace{.7pt}\log\hspace{-1.4pt}{}^{#1}\hspace{-0.25pt}#2}}
\def\disc{\operatorname{disc}}

\def\tr{\operatorname{tr}}

\def\Exp{\textbf{E}}
\newcommand{\idx}[2]{[#1\hspace{1.5pt}\text{\rm :}\hspace{2pt}#2]}

\newtheorem{proposition}{Proposition}
\newtheorem{definition}{Definition}
\newtheorem{corollary}{Corollary}

\newcommand{\algstart}[2]{\smallskip\noindent{\bf Algorithm} #1. \emph{#2}\begin{enumerate}}
\newcommand{\algend}{\end{enumerate}\vspace{4pt}}
\newcommand{\algitem}{\vspace{2pt}\item}

\begin{document}
\title{Identifying supersingular elliptic curves}

\author{Andrew V. Sutherland}
\address{Department of Mathematics, Massachusetts Institute of Technology, Cambridge, Massachusetts 02139}
\email{drew@math.mit.edu}
\thanks{The author received financial support from NSF grant DMS-1115455}


\begin{abstract}
Given an elliptic curve $E$ over a field of positive characteristic~$p$, we consider how to efficiently determine whether $E$ is ordinary or supersingular.
We analyze the complexity of several existing algorithms and then present a new approach that exploits structural differences between ordinary and supersingular isogeny graphs.
This yields a simple algorithm that, given $E$ and a suitable non-residue in $\Fptwo$, determines the supersingularity of $E$ in $O(n^3 \logpow{2}{n})$ time and $O(n)$ space, where $n=O(\log p)$.
Both these complexity bounds are significant improvements over existing methods, as we demonstrate with some practical computations.
\end{abstract}

\maketitle

\section{Introduction}
An elliptic curve $E$ over a field $F$ of positive characteristic~$p$ is called \emph{supersingular} if its $p$-torsion subgroup $E[p]$ is trivial; see \cite[\S 13.7]{Husemoller:EllipticCurves} or \cite[\S V.3]{Silverman:EllipticCurves1} for several equivalent definitions.  Otherwise, we say that $E$ is \emph{ordinary}.
Supersingular curves differ from ordinary curves in many ways, and this has practical implications for algorithms that work with elliptic curves over finite fields, such as algorithms for counting points~\cite{Schoof:ECPointCounting2},
generating codes~\cite{Schoof:CodingTheory}, computing endomorphism rings~\cite{Kohel:thesis}, and calculating discrete logarithms~ \cite{Menezes:MOV}.
Given an elliptic curve, one of the first things we might wish to know is whether it is ordinary or supersingular, and we would like to make this distinction as efficiently as possible.

The answer to this question depends only on the isomorphism class of $E$ over~$\bar{F}$, which is characterized by its $j$-invariant $j(E)$.
It is known that $E$ can be supersingular only when $j(E)\in\Fptwo$, thus we may restrict our attention to the case that~$F$ is a finite field $\Fq\subseteq\Fptwo$.
We also recall that $E$ is supersingular if and only if $\#E(\Fq)\equiv 1 \bmod p$; see \cite{Silverman:EllipticCurves1} for proofs of these facts.

There is a simple Monte Carlo test that quickly identifies ordinary elliptic curves.
When $q=p$, one picks a random point~$P$ on the curve and computes the scalar multiple $(p+1)P$.
If $(p+1)P\ne 0$ then the curve is ordinary, and if $(p+1)P=0$ then the curve is likely to be supersingular (see \S \ref{subsection:montecarlo} for the case $q=p^2$).
If several repetitions of this test fail to prove that $E$ is ordinary, then it is almost certainly supersingular.
But this approach cannot \emph{prove} that $E$ is supersingular, just as the Miller-Rabin primality test~\cite{Miller:PrimalityTest} cannot prove that an integer is prime.

To prove that $E$ is supersingular, one may verify that \mbox{$\#E(\Fq)\equiv 1\bmod p$} using a point-counting algorithm, such as Schoof's algorithm \cite{Schoof:ECPointCounting1,Schoof:ECPointCounting2}.
With a variant of the SEA algorithm (see \S \ref{subsection:polytime}), this can be accomplished in $O(n^4\llog n)$ time using $O(n^4)$ space, where $n=\log q$.
The computer algebra systems Magma~\cite{Magma} and Sage~\cite{SAGE} both use this approach to identify supersingular curves.

But it is natural to ask whether one can do better.
We show that this is indeed the case, presenting an algorithm that runs in ${O}(n^3\logpow{2}{n})$ time and $O(n)$ space.
Rather than counting points, we rely on structural differences between ordinary and supersingular isogeny graphs.
The resulting algorithm is easy to implement and much faster than methods based on point counting, as may be seen in Table~\ref{table:timings}.

In the first step of the algorithm we must solve a cubic equation, and in each subsequent step we need to solve a quadratic equation.
To obtain a deterministic result, we assume that we are given a quadratic non-residue and a cubic non-residue in~$\Fptwo$ to facilitate these computations.
When~$\Fptwo$ is constructed using a generator, this generator already provides the non-residues we require.
Alternatively, non-residues can be efficiently obtained by sampling random elements, yielding a Las Vegas algorithm.

\section{Existing Algorithms}\label{section:background}

Before presenting the new algorithm, we briefly review some existing methods for testing supersingularity and analyze their complexity.
Over fields of characteristic~$2$ or~$3$, an elliptic curve $E$ is supersingular if and only if $j(E)=0$, a condition that is trivial to check given an equation for the curve.
As noted in the introduction, we may assume that $E$ is defined over $\Fptwo$, since otherwise $E$ is ordinary.
Thus we shall work over a finite field $\Fq$ of characteristic $p > 3$, where $q$ is either $p$ or $p^2$.

We use $\M(n)$ to denote the cost of multiplying two $n$-bit integers, which we may bound by $\M(n)=O(n\log n\llog n)=\tilde{O}(n)$, via \cite{Schonhage:Multiplication}.
All of our bounds are expressed in terms of $n=\log p$, which is proportional to the size of the input for our problem, the coefficients of the curve $E$.

\subsection{Exponential time algorithms}\label{subsection:exptime}

If $E$ is in Weierstrass form $y^2=f(x)$, then $E$ is supersingular if and only if the coefficient of $x^{p-1}$ in $f(x)^{(p-1)/2}$ is zero, and this implies that if $E$ is in Legendre form $y^2=x(x-1)(x-\lambda)$, then $E$ is supersingular if and only if $\sum_{i=0}^m\binom{m}{i}^2\lambda^i = 0$; see \cite[Thm.~4.1]{Silverman:EllipticCurves1}.
These criterion are convenient and easy to state, but they are computationally useful only when $p$ is very small, since the time required to apply them is exponential in $n$.

\subsection{Polynomial time algorithms}\label{subsection:polytime}

Schoof's algorithm \cite{Schoof:ECPointCounting1} computes $\#E(\Fq)$ in $\tilde{O}(n^5)$ time and $O(n^3)$ space.
This immediately yields a deterministic polynomial-time algorithm for testing supersingularity, since $E$ is supersingular if and only if $\#E(\Fq)\equiv 1\bmod p$.
The improvements of Elkies and Atkin incorporated in the SEA algorithm \cite{Elkies:AtkinBirthday,Schoof:ECPointCounting2} are not immediately applicable, since they rely on results that do not necessarily apply to supersingular curves \cite[Prop.~6.1-3]{Schoof:ECPointCounting2}.
However, as remarked by Schoof \cite[p.~241]{Schoof:ECPointCounting2}, supersingular curves can be identified using similar techniques.
Let us briefly fill in the details.

Recall that for any prime $\ell\ne p$, the classical modular polynomial $\Phi_\ell\in\ZZ[X,Y]$ has the property that two $j$-invariants $j_1,j_2\in\Fq$ satisfy $\Phi_\ell(j_1,j_2) = 0$ if and only if $j_1=j(E_1)$ and $j_2=j(E_2)$ for some elliptic curves $E_1$ and $E_2$ related by a cyclic isogeny of degree $\ell$; see \cite[Thm.~12.19]{Lang:EllipticFunctions}.
If $E_1$ and $E_2$ are isogenous, then $\#E_1(\Fq)=\#E_2(\Fq)$, thus $E_1$ is supersingular if and only if $E_2$ is.
Since every supersingular $j$-invariant in characteristic $p$ lies in $\Fptwo$, if $E$ is supersingular then the univariate polynomial $\phi_{\ell,E}(X) = \Phi_\ell(j(E),X)$ splits completely in $\Fptwo[X]$, for every prime $\ell\ne p$.  However, if $E$ is ordinary, this is not the case.

\begin{proposition}\label{prop:modpoly}
Let $j(E)\in\Fptwo$ and assume $j(E)\ne 0,1728$.\footnote{We note that $j(E)=0$ (resp. 1728) is supersingular if and only if $p\not\equiv 1 \bmod 3$ (resp. 4).}
Let $S$ be a set of primes $\ell\ne p$ with product $M>2p$.
Then $E$ is supersingular if and only if $\phi_{\ell,E}$ splits completely in $\Fptwo[X]$ for every $\ell\in S$.
\end{proposition}
\begin{proof}
The forward implication is addressed by the discussion above.
For the reverse, suppose for the sake of contradiction that $E$ is ordinary and that $\phi_{\ell,E}$ splits completely in $\Fptwo[X]$ for all $\ell\in S$.
It follows from \cite[Thm.~2.1]{Fouquet:IsogenyVolcanoes} (or see \S\ref{sec:isogenies}) that $t^2-4p^2$ is divisible by $\ell^2$, where $t=p^2+1-\#E(\Fptwo)$ is the trace of Frobenius of $E/\Fptwo$.
Thus $t^2\equiv 4p^2\bmod \ell^2$ for each $\ell\in S$, and therefore $t^2\equiv 4p^2\bmod M^2$, by the Chinese Remainder Theorem.
The Hasse bound implies $t^2\le 4p^2$, so we must have $t^2 = 4p^2$, since $M^2 > 4p^2$.
Thus $t=\pm 2p$, and therefore $\#E(\Fptwo)\equiv 1\bmod p$.  But this implies that $E$ is supersingular, which is a contradiction.
\end{proof}

To prove the supersingularity of $E/\Fq$, it is enough to check that $\phi_{\ell,E}$ splits completely in $\Fptwo[X]$ for each of the first $m$ primes $\ell$ with product $M > 2p$.
%\footnote{The bound $M > 4\sqrt{q}$ does \emph{not} suffice: consider $q=3539$, $j(E)=922$, and $M=2\cdot 3\cdot 5\cdot 7\cdot 11$.}
This can be done without factoring $\phi_{\ell,E}$.  One removes all linear factors from $\phi_{\ell,E}$ as follows: first let $f = \phi_{\ell,E}$ and compute $g = \gcd(f(X),X^{p^2}-X)$, then repeatedly set $f\leftarrow f/g$ and $g\leftarrow \gcd(f,g)$ until $\deg g = 0$.
If at this point $\deg f = 0$, then $\phi_{\ell,E}$ splits completely over $\Fptwo$ and otherwise it does not.
When $j(E)$ lies in $\Fp$, we may instead work in $\Fp[X]$ and remove both linear and quadratic factors from $\phi_{\ell,E}$ with a similar approach.

Using precomputed modular polynomials, this yields a deterministic algorithm that runs in $O(n^2\M(n^2)/\log n)=O(n^4\llog n)$ time and $O(n^4)$ space,
assuming Kronecker substitution \cite[\S 8.4]{Gathen:ComputerAlgebra} is used to multiply polynomials in $\Fptwo[X]$ of degree $O(n)$ in time $O(\M(n^2))$.
The space can be reduced to $O(n^3\log n)$ by computing modular polynomials as required, but this significantly increases the running time.

\subsection{A Monte Carlo algorithm}\label{subsection:montecarlo}

For a supersingular curve $E$ over a field of characteristic $p > 3$ it follows from \cite{Ruck:EllipticCurveGroupStructures} that
\begin{enumerate}
\item[(i)] if $E$ is defined over $\Fp$ then $\#E(\Fp)=p+1$;
\item[(ii)] either $E(\Fptwo)\cong (\ZZ/(p-1)\ZZ)^2$ or $E(\Fptwo)\cong (\ZZ/(p+1)\ZZ)^2$.
\end{enumerate}
This motivates the following algorithm.
\smallskip

\algstart{\textbf{1}}{Given an elliptic curve $E/\Fq$ with $q|p^2$, do the following:}
\algitem {\bf If $\boldsymbol{q=p}$:} pick a random point~$P\in E(\Fp)$ and return {\bf true} if $(p+1)P=0$, otherwise return {\bf false}.
\algitem {\bf If $\boldsymbol{q=p^2}$:} pick a random point~$P\in E(\Fptwo)$ and return {\bf true} if either $(p-1)P=0$ or $(p+1)P=0$, otherwise return {\bf false}.
\algend

If the algorithm returns {\bf false} then $E$ is ordinary.
We now show that if the algorithm returns {\bf true}, then $E$ is very likely to be supersingular (for large $q$).

\begin{proposition}\label{prop:montecarlo}
Given an ordinary elliptic curve $E/\Fq$, Algorithm~1 returns {\bf true} with probability at most $8\sqrt{q}/(\sqrt{q}-1)^2 = O(q^{-1/2})$.
\end{proposition}
\begin{proof}
First, let $q=p$.
Let $H$ be the $(p+1)$-torsion subgroup $E(\Fq)[p+1]$.
Then $H\cong \ZZ/m_1\ZZ\times \ZZ/m_2\ZZ$, where $m_1$ divides $m_2$ and $q-1$.
Since $m_1$ also divides $p+1$, we have $m_1\le 2$.
We now show $m_2\le 4\sqrt{q}$.  If not, then $p+1$ is the unique multiple of $m_2$ in the Hasse interval $[(\sqrt{q}-1)^2,(\sqrt{q}+1)^2]$.
But then $\#E(\Fp)=p+1$, contradicting the fact that $E$ is ordinary.
Thus $\#H=m_1m_2\le 8\sqrt{q}$.

Now let $q=p^2$.
Let $H$ be the union of $H_1=E(\Fq)[p-1]$ and $H_2=E(\Fq)[p+1]$.
Then $\#H_1\le 4\sqrt{q}$, else $(p-1)^2$ is the unique multiple of $\#H_1$ in the Hasse interval, yielding a contradiction as above.
Similarly, $\#H_2\le 4\sqrt{q}$, and therefore $\#H\le 8\sqrt{q}$.

In both cases, Algorithm~1 outputs ${\bf true}$ only when the random point~$P$ lies in~$H$, which occurs with probability $\#H/\#E(\Fq)\le 8\sqrt{q}/(\sqrt{q}-1)^2$.
\end{proof}

Algorithm~1 is a Monte Carlo algorithm with one-sided error.
For $q\ge 7$ the error probability given by Proposition~\ref{prop:montecarlo} is bounded below 1 and can be made arbitrarily small (but never zero) by repetition.
Using standard techniques, the random point~$P$ can be obtained in $O(n\M(n))=\tilde{O}(n^2)$ expected time, and this also bounds the cost of the scalar multiplications.
%The algorithm uses $O(n)$ space.

\section{Isogeny graphs}\label{sec:isogenies}
As above, we work in a finite field $\Fq$ of characteristic $p > 3$.
For each prime $\ell\ne p$ we define the (directed multi-) graph $G_\ell(\Fq)$ of $\Fq$-rational $\ell$-isogenies.

\begin{definition}
$G_\ell(\Fq)$ is the graph with vertex set $\Fq$ and edges $(j_1,j_2)$ present with multiplicity $k$ whenever $j_2$ is a root of $\Phi_\ell(j_1,X)$ with multiplicity~$k$.
\end{definition}

As in \S \ref{subsection:polytime}, the polynomial $\Phi_\ell\in\ZZ[X,Y]$ is the classical modular polynomial that parametrizes $\ell$-isogenous pairs of $j$-invariants; see \cite[\S 5.2]{Lang:EllipticFunctions}.
It is symmetric and has degree $\ell+1$ in both variables, thus the in-degree and out-degree of each vertex of $G_\ell(\Fq)$ is at most $\ell+1$.
These degrees need not coincide (e.g., for the vertices 0, 1728, and their neighbors); when we speak of the \emph{degree} of a vertex we refer to its out-degree.
We note that $G_\ell(\Fq)$ may contain self-loops, edges of the form $(j_1,j_1)$.

Each vertex of $G_\ell(\Fq)$ is the $j$-invariant $j(E)$ of an elliptic curve $E$ defined over~$\Fq$, and we may classify each vertex as ordinary or supersingular.
We may similarly classify the edges and connected components of $G_\ell(\Fq)$, since every edge lies between vertices of the same type (ordinary or supersingular).
As noted in \S \ref{subsection:polytime}, if $j(E)$ is a supersingular $j$-invariant then the polynomial $\Phi_l(j(E),X)$ splits completely in $\Fptwo[X]$, and it follows that for $q>p$, every supersingular component\footnote{There is in fact only one supersingular component of $G_\ell(\Fptwo)$; see \cite[Cor.~78]{Kohel:thesis}.} of $G_\ell(\Fq)$ is a regular graph of degree $\ell+1$.

However, the ordinary components of $G_\ell(\Fq)$ are \emph{not} regular graphs of degree $\ell+1$; they contain many vertices of degree less than $\ell+1$, and this is the basis of our algorithm.
Given an elliptic curve $E$ defined over $\Fptwo$, our strategy is to search for a vertex of degree less than 3 that is connected to $j(E)$ in $G_2(\Fptwo)$.
If we find such a vertex, then $E$ is ordinary, and if we can prove no such vertex exists, then $E$ is supersingular.
To do this we need to understand the structure of the ordinary components of $G_2(\Fptwo)$.
All the facts we require apply more generally to $G_\ell(\Fq)$, so we continue in this setting.

A detailed analysis of the structure of the ordinary components of $G_\ell(\Fq)$ was undertaken by Kohel in his thesis \cite{Kohel:thesis}, and they are now commonly called $\ell$-\emph{volcanoes}, a term introduced by Fouquet and Morain \cite{Fouquet:IsogenyVolcanoes}.
The structure of an $\ell$-volcano is determined by the relationships between the endomorphism rings of the elliptic curves corresponding to its vertices.
Here we record only the facts we need, referring to \cite{Fouquet:IsogenyVolcanoes,Kohel:thesis} for proofs and a more complete presentation.

Let $j(E)$ be a vertex in an ordinary component $V$ of $G_\ell(\Fq)$ (an $\ell$-volcano).
Recall that the endomorphism ring of an ordinary elliptic curve is isomorphic to an order $\OO$ in an imaginary quadratic field $K$.
We have the inclusions $\ZZ[\pi]\subseteq \OO\subseteq\OO_K$, where $\ZZ[\pi]$ is the order generated by (the image of) the Frobenius endomorphism~$\pi$, and $\OO_K$ is the maximal order of $K$ (its ring of integers).
The order $\OO$ depends only on the isomorphism class $j(E)$, while the orders $\ZZ[\pi]$ and $O_K$ depend only on the isogeny class of $E$ and are invariants of $V$.

We may partition the vertices of $V$ into \emph{levels} $V_0,\ldots,V_d$, where the level $V_i$ in which $j(E)$ lies is determined by the $\ell$-adic valuation $i=\nu_\ell(\idx{\OO_K}{\OO})$, the largest integer $i$ for which $\ell^i$ divides $\idx{\OO_K}{\OO}$.
The integer $d=\nu_\ell\idx{\OO_K}{\ZZ[\pi]}$ is the \emph{depth} (also called the \emph{height}) of $V$, and may be 0.
From the norm equation
\begin{equation}\label{eq:norm}
4q=t^2-v^2D,
\end{equation}
where $q=N(\pi)$, $t=\tr\pi$, $D=\disc(K)$, and $d =\nu_\ell(v)$, we have
\begin{equation}\label{eq:dbound}
d < \log_\ell \sqrt{4q}.
\end{equation}

Level $V_d$ is the \emph{floor} of the $\ell$-volcano $V$.
Its vertices are distinguished by their degree, which is at most 2.
Every other vertex in $V$ (if any) has degree $\ell+1$.

\begin{proposition}\label{prop:volcano}
Let $j(E)$ be a vertex in level $V_i$ of an $\ell$-volcano $V$ of depth $d$.
\begin{enumerate}
\item[(i)] The degree of $j(E)$ is $\ell+1$ if and only if $i < d$.
\item[(ii)] If $i = 0 < d$ then at least $\ell-1$ of the edges from $j(E)$ lead to $V_1$.
\item[(iii)] If $0 < i < d$ then one edge from $j(E)$ leads to $V_{i-1}$ and the rest lead to $V_{i+1}$.
\item[(iv)] If $0 < i = d$ then $j(E)$ has just one outgoing edge and it leads to $V_{d-1}$.
\end{enumerate}
\end{proposition}
\begin{proof} See  \cite[Thm.~2.1]{Fouquet:IsogenyVolcanoes} and \cite[Prop.~23]{Kohel:thesis}.
\end{proof}

Given $E/\Fq$, our goal is to either find a path from $j(E)$ to the floor of its $\ell$-volcano in $G_\ell(\Fq)$, or prove that no such path exists.
We define a path as follows.

\begin{definition}
A path (of length $k$) in $G_\ell(\Fq)$ is a sequence of vertices $j_0,j_1,\ldots,j_k$ such that $\Phi_\ell(j_0,j_1)=0$ and $j_{i+2}$ is a root of $\Phi_\ell(j_{i+1},X)/(X-j_i)$ for $0\le i < k-1$.
\end{definition}

In terms of a walk on the graph, this definition prohibits backtracking except when there are multiple edges leading back to the previous vertex.
Edges that lead toward the floor (from level $V_i$ to $V_{i+1}$) are called \emph{descending}.
Proposition \ref{prop:volcano} implies that every vertex of $V$ not on the floor has at least $\ell-1$ descending edges.
Any path that starts with a descending edge can only be extended by descending further, and this must lead to the floor within $d$ steps (this is called a \emph{descending path} in \cite{Fouquet:IsogenyVolcanoes}).

We can summarize these results in purely graph-theoretic terms.
For any edge $(j_0,j_1)$ in $G_\ell(\Fq)$, not necessarily ordinary, let $R_k(j_0,j_1)$ denote the set of vertices~$j_{k}$ for which there exists a path $j_0,j_1,\ldots,j_k$ of length $k$.

\begin{corollary}\label{corollary:paths}
Let $j_0$ be a vertex of $G_\ell(\Fq)$ of degree $\ell+1$.
\begin{enumerate}
\item[(i)] If $j_0$ is ordinary, then $G_\ell(\Fq)$ contains $\ell-1$ edges $(j_0,j_1)$ for which the set $R_k(j_0,j_1)$ is empty for some $1\le k < \log_\ell\sqrt{4q} + 1$.
\item[(ii)] If $j_0$ is supersingular and $q > p$, then for every edge $(j_0,j_1)$ the set $R_k(j_0,j_1)$ is nonempty for all $k\ge 1$.
\end{enumerate}
\end{corollary}
 
\section{The algorithm}

We now present our algorithm, which, given an elliptic curve over a field of positive characteristic, returns ${\bf true}$ if $E$ is supersingular and ${\bf false}$ otherwise.
\pagebreak

\algstart{\textbf{2}}{Given an elliptic curve $E/F$ with $\operatorname{char} F = p > 0$, do the following:}
\algitem If $j(E)\notin\Fptwo$ then return {\bf false}.
\algitem If $p\le 3$ then return {\bf true} if $j(E)=0$ and {\bf false} otherwise.
\algitem Attempt to find three roots $j_1,j_2,j_3$ of $\Phi_2(j(E),X)$ in $\Fptwo$.\\
If $\Phi_2(j(E),X)$ does not have three roots in $\Fptwo$ then return {\bf false}.
\algitem Set $j'_i \leftarrow j(E)$ for $i=1,2,3$.
\algitem Let $m=\lfloor\log_2 p\rfloor+1$, and for $k=1$ to $m$:
\begin{enumerate}
\item Set $f_i(X)\leftarrow \Phi_2(j_i,X)/(X-j'_i)$ and set $j'_i\leftarrow j_i$, for $i=1,2,3$.
\item Attempt to find a root $j_i$ of $f_i(X)$ in $\Fptwo$, for $i=1,2,3$.\\
If any $f_i(X)$ does not have a root in $\Fptwo$ then return {\bf false}.
\end{enumerate}
\algitem Return {\bf true}.
\algend

After ruling out some trivial cases, the algorithm begins in step 3 by computing the outgoing edges from the vertex $j(E)$ in $G_2(\Fptwo)$,
using the modular polynomial
\begin{align*}
\Phi_2(X,Y) = X^3 &+ Y^3 - X^2Y^2 + 1488(X^2Y+Y^2X) - 162000(X^2+Y^2)\\
 &+ 40773375XY + 8748000000(X+Y) - 157464000000000.
\end{align*}
If the vertex $j(E)$ does not have degree 3 then $E$ must be ordinary and the algorithm terminates.
Otherwise, it attempts to extend each of the three edges $(j(E),j_i)$ to a path of length $m+1 \ge \log_2 \sqrt{4p^2} + 1$ in step~5.
If $E$ is ordinary than one of these attempts must fail, and otherwise $E$ must be supersingular, by Corollary \ref{corollary:paths}.

Thus the algorithm is correct.
We now analyze its complexity, considering two possible implementations, one probabilistic and one deterministic.
As in \S\ref{section:background}, we let $\M(n)$ denote the cost of multiplication and express our bounds in terms of $n=\log p$.

\subsection{Probabilistic complexity analysis}
The work of Algorithm~2 consists essentially of solving a cubic equation in step~3 and at most $3m=O(n)$ quadratic equations in step~5.
With a probabilistic root-finding algorithm \cite[Alg~14.5]{Gathen:ComputerAlgebra}, we expect to use $O(n)$ operations in $\Fptwo$ for each equation,
yielding a total expected running time of $O(n^2)$ operations in $\Fptwo$, using storage for $O(1)$ elements of~$\Fptwo$.
This gives an expected running time of $O(n^2\M(n))= O(n^3\log n\llog n)$ using $O(n)$ space.
The output of the algorithm is not affected by any of the random choices that are made (it is always correct), thus we have a Las Vegas algorithm.

\begin{proposition}\label{prop:complexityLV}
Algorithm~2 can be implemented as a Las Vegas algorithm with an expected running time of $O(n^3\log n\llog n)$, using $O(n)$ space.
\end{proposition}

\subsection{Deterministic complexity analysis}
We now consider how we may obtain a deterministic algorithm, given some additional information.
First, we note that the choice of the root $j_i$ in step~5 can be fixed by ordering $\Fptwo$ with respect to some basis.
Second, we may apply the quadratic formula and Cardano's method (valid over any field of characteristic not 2 or 3), to solve the equations arising in steps~3 and~5 by radicals.
To find the roots of a quadratic or cubic polynomial that splits completely in $\Fptwo[X]$, it suffices to compute square roots and cube roots in $\Fptwo$.
 
For any prime $r$, computing an $r$th root in a finite field $\Fq$ can be reduced to an exponentiation and a (possibly trivial) discrete logarithm computation in the $r$-Sylow subgroup of $\Fq^*$.
For $r=2$ this is the Tonelli-Shanks algorithm~\cite{Tonelli:SquareRoot,Shanks:FiveAlgorithms}, and the generalization to $r>2$ is due to Adleman, Manders, and Miller~\cite{Adleman:Roots}.
For the discrete logarithm computation we require a generator $\gamma$ for the $r$-Sylow subgroup~$H_r$ of~$\Fq^*$ (which is necessarily cyclic).
Using the algorithm in \cite{Sutherland:AbelianpGroups} we can compute discrete logarithms in $H_r$ using $O(n\log n / \llog n)$ operations in~$\Fq$, assuming $r$ and the degree of $\Fq$ are fixed.
This yields a bit-complexity of $O(\M(n)n\log n/\llog n) = O(n^2\logpow{2}{n})$, which dominates the cost of exponentiation.

When $H_r$ is not trivial, any element $\alpha$ of $\Fq$ that is not an $r$th-power residue yields a generator for $H_r$: simply let $\gamma = \alpha^{(q-1)/s}$, where $s=r^{\nu_r(q-1)}$.
This yields the following proposition.

\begin{proposition}\label{prop:complexityDet}
Algorithm~2 can be implemented as a deterministic algorithm that runs in $O(n^3\logpow{2}{n})$ time using $O(n)$ space, given a quadratic non-residue and a cubic non-residue in $\Fptwo$.
\end{proposition}

As noted earlier, we can efficiently obtain non-residues by sampling random elements.
Given a uniformly random $\alpha\in\Fq^*$, if we let $\gamma = \alpha^{(q-1)/s}$ as above, then~$\gamma$ generates $H_r$ if and only if $\gamma^{s/r}\ne 1$, which occurs with probability $1-1/r$.
Alternatively, if we are given a generator for $\Fptwo$ (the coefficients of $E$ may be specified in terms of such a generator), then we already have an element that is both a quadratic and a cubic non-residue.

We remark that while the complexity bound in Proposition \ref{prop:complexityDet} is slightly worse than the bound in Proposition \ref{prop:complexityLV}, in practice the deterministic approach is usually faster; the 2-Sylow and 3-Sylow subgroups of most finite fields are very small, and in this case the discrete logarithms used to compute square roots and cube roots take negligible time.

\subsection{Average case complexity}
The bounds given in Propositions \ref{prop:complexityLV} and \ref{prop:complexityDet} are worst-case complexity bounds.
We now consider the performance of Algorithm 2, on average, when given a random elliptic curve over $\Fptwo$.

\begin{proposition}\label{prop:complexityAve}
Given an elliptic curve whose $j$-invariant is uniformly distributed over $\Fptwo$, the expected running time of Algorithm 2 is $O(n^2\log n\llog n)$.
\end{proposition}
\begin{proof}
By \cite[Thm.~4.1]{Silverman:EllipticCurves1}, the proportion of supersingular $j$-invariants in $\Fptwo$ is $O(1/p)$.
It follows from Propositions~\ref{prop:complexityLV} and~\ref{prop:complexityDet} that these cases have a negligible impact on the expected running time.
Given an ordinary elliptic curve with $j$-invariant $j_0$, the running time of Algorithm 2 is $O(n\Exp[d-i+1])$ field operations, where $d$ is the depth of the 2-volcano in $G_2(\Fptwo)$ containing $j_0$, and $V_i$ is the level in which $j_0$ lies.
By Proposition~\ref{prop:volcano}, for $d > 0$ we have $\#V_0 \le \#V_1$ and $\#V_i = \#V_{i+1}/2$, for $0 < i < d$.
This implies that $\Exp[d-i+1]$ is $O(1)$, and the proposition follows.
\end{proof}

The bound in Proposition \ref{prop:complexityAve} applies to both the probabilistic and deterministic implementations of Algorithm 2 considered above.
With a probabilistic implementation, the expected running time of Algorithm 2 is within a constant factor of the running time of the Monte Carlo approach used in Algorithm 1,
and for almost all values of $p$ (those for which $p^2-1$ is not divisible by an unusually large power of 2 or 3), this is also true of the deterministic implementation.
Remarkably, this constant factor actually favors Algorithm~2, which identifies most ordinary curves even more quickly than Algorithm~1 (see Table~\ref{table:timings}).

\section{Computational results}\label{section:computations}

Table \ref{table:timings} compares the performance of Algorithm~2 with the implementation of the \textsc{IsSupersingular} function provided by the Magma computer algebra system.
The Magma implementation relies on two standard methods for distinguishing supersingular curves: it first performs a Monte Carlo test to quickly identify ordinary curves (as in Algorithm~1), and then applies the modular polynomial approach described in \S \ref{subsection:polytime}.
Our implementation was built on the GNU Multiple Precision Arithmetic Library (GMP) \cite{GMP}, which is also used by Magma.
All tests were run on a single core of an AMD Opteron 250 processor clocked at 2.4 GHz.

\begin{table}
\caption{Performance results (CPU times in milliseconds).}\label{table:timings}
\begin{tabular}{rrrrrrrrrrrrr}
&&\multicolumn{5}{c}{ordinary}&&\multicolumn{5}{c}{supersingular}\\
\cmidrule(r){3-7}\cmidrule(r){9-13}
&&\multicolumn{2}{c}{Magma}&&\multicolumn{2}{c}{Alg.~2}&&\multicolumn{2}{c}{Magma}&&\multicolumn{2}{c}{Alg.~2}\\
\cmidrule(r){3-4}\cmidrule(r){6-7}\cmidrule(r){9-10}\cmidrule(r){12-13}
$b$&& $\Fp$ & $\Fptwo$ && $\Fp$ & $\Fptwo$ &&  $\Fp$ & $\Fptwo$ && $\Fp$ & $\Fptwo$\\
\midrule
  64 &&    1 &   25 && 0.1 & 0.1 &&      226 &      770 &&    2 &    8\\
 128 &&    2 &   60 && 0.1 & 0.1 &&     2010 &     9950 &&    5 &   13\\
 192 &&    4 &   99 && 0.2 & 0.1 &&     8060 &    41800 &&    8 &   33\\
 256 &&    7 &  140 && 0.3 & 0.2 &&    21700 &   148000 &&   20 &   63\\
 320 &&   10 &  186 && 0.4 & 0.3 &&    41500 &   313000 &&   39 &  113\\
 384 &&   14 &  255 && 0.6 & 0.4 &&    95300 &   531000 &&   66 &  198\\
 448 &&   19 &  316 && 0.8 & 0.5 &&   152000 &   789000 &&  105 &  310\\
 512 &&   24 &  402 && 1.0 & 0.7 &&   316000 &  2280000 &&  164 &  488\\
 576 &&   30 &  484 && 1.3 & 0.9 &&   447000 &  3350000 &&  229 &  688\\
 640 &&   37 &  595 && 1.6 & 1.0 &&   644000 &  4790000 &&  316 &  945\\
 704 &&   46 &  706 && 2.0 & 1.2 &&   847000 &  6330000 &&  444 & 1330\\
 768 &&   55 &  790 && 2.4 & 1.5 &&  1370000 &  8340000 &&  591 & 1770\\
 832 &&   66 &  924 && 3.1 & 1.9 &&  1850000 & 10300000 &&  793 & 2410\\
 896 &&   78 & 1010 && 3.2 & 2.1 &&  2420000 & 12600000 && 1010 & 3040\\
 960 &&   87 & 1180 && 4.0 & 2.5 &&  3010000 & 16000000 && 1280 & 3820\\
1024 &&  101 & 1400 && 4.8 & 3.1 &&  5110000 & 35600000 && 1610 & 4880\\
\bottomrule
\end{tabular}
\end{table}

Each row of Table~\ref{table:timings} corresponds to a series of tests using a fixed bit-length~$b$.
For each value of $b$ we selected 5 random primes $p$ in the interval $[2^{b-1}, 2^b]$, and for each prime $p$ we generated 100 elliptic curves defined over $\Fp$ and 100 elliptic curves defined over $\Fptwo$, with uniformly distributed $j$-invariants.
As one might expect, all of these randomly generated curves were ordinary, and the average times to process these curves are listed in the ``ordinary" columns of Table~\ref{table:timings}.

To test performance on supersingular inputs, for each prime $p$ we constructed a supersingular curve over $\Fp$ using a variant of the CM method described in \cite{Broker:ConstructingSupersingularCurves}.
This involves picking a discriminant $D < 0$ with $\inkron{D}{p}=-1$ and $-D$ prime.
The Hilbert class polynomial $H_D(X)$ is then guaranteed to have an $\Fp$-rational root $j_0$, which is necessarily the $j$-invariant of a supersingular elliptic curve.
In order for this to be feasible, the discriminant $D$ cannot be too large; we used random discriminants in the interval $[2^{31},2^{32}]$, and computed $H_D(X)\bmod p$ using the algorithm in \cite{Sutherland:HilbertClassPolynomials}.

Over $\Fp$, the supersingular $j$-invariants obtained in this fashion are not uniformly distributed over the set of supersingular $j$-invariants in $\Fp$.
However, one expects the running times of both Algorithm 2 and the Magma implementation to be essentially independent of $D$, and this appears to be the case.
Over $\Fptwo$, we are able to obtain a nearly uniform distribution of supersingular $j$-invariants by performing a random walk on the graph $G_2(\Fptwo)$, starting from a vertex defined over $\Fp$ constructed using the CM method described above.
The supersingular component~$S$ of~$G_2(\Fptwo)$ is a Ramanujan graph \cite{Pizer:RamanujanGraphs}, and this implies that, starting from any vertex of $S$, a random walk of $O(n)$ steps on $S$ yields a nearly uniform distribution on its vertices.

\subsection{Discussion of results}
Table~1 indicates a significant performance advantage for Algorithm~2, both asymptotically (as predicted by the complexity analysis), and in terms of its constant factors.
It is worth noting that for both ordinary and supersingular inputs, the Magma implementation is substantially slower when working over $\Fptwo$ rather than $\Fp$.
This is to be expected, given the higher cost of finite field operations in $\Fptwo$.
By contrast, Algorithm~2 always works in $\Fptwo$, and one might suppose that its performance should be essentially independent of whether the input curves is defined over $\Fp$ or $\Fptwo$.
As can be seen in the timings in Table~1, this is not quite the case.  There are two reasons for this.

First, for a random elliptic curve $E/\Fptwo$, the probability that the vertex $j(E)$ has degree 3 in $G_2(\Fptwo)$ is, asymptotically, only $1/6$.
This means that in approximately 5/6 of the cases (whenever $\phi_{\ell,E}(X)$ does not split completely in $\Fptwo[X]$), Algorithm~2 terminates in step~3.
But if we restrict to $E/\Fp$, this happens in just $1/3$ of the cases (namely, whenever $\phi_{\ell,E}(X)$ is irreducible in $\Fp[X]$).
This difference explains why Algorithm~2 is actually somewhat faster, on average, when given a random curve over $\Fptwo$ rather than $\Fp$.

Second, our implementation relies on a practical optimization that can be applied whenever the input curve is defined over $\Fp$, and this optimization yields nearly a 3-fold speedup on supersingular inputs.
Rather than working entirely in the graph $G_2(\Fptwo)$, we begin by searching for a path in $G_2(\Fp)$ from $j(E)$ to a vertex of degree~1, walking three paths in parallel as usual.
Such a vertex $j_i$ will will be found within $O(1)$ steps, on average.
The vertex $j_i$ will necessarily have degree 3 in $G_2(\Fptwo)$, and if $E$ is ordinary, then the two edges that lead from $j_i$ to vertices that are not defined over $\Fp$ must be \emph{descending} edges.
It then suffices to extend just one path containing one of these edges, rather than walking three paths in parallel.

\section{Acknowledgements}

I am grateful to David Kohel for his feedback on an early draft of this paper, and for showing how to tighten the bound in Proposition~\ref{prop:modpoly}.

\bibliographystyle{amsplain}
%\bibliography{../general}
\providecommand{\bysame}{\leavevmode\hbox to3em{\hrulefill}\thinspace}
\providecommand{\MR}{\relax\ifhmode\unskip\space\fi MR }
\providecommand{\MRhref}[2]{%
  \href{http://www.ams.org/mathscinet-getitem?mr=#1}{#2}
}
\providecommand{\href}[2]{#2}

\end{document}